\documentclass[review,11pt]{elsarticle}

\usepackage[margin=3cm]{geometry}
\usepackage{lineno,hyperref}
\usepackage{amssymb,amsmath,amsthm,latexsym,xcolor}
\usepackage{enumitem}

\let\oldequation\equation
\let\oldendequation\endequation

\renewenvironment{equation}
  {\linenomathNonumbers\oldequation}
  {\oldendequation\endlinenomath}
\journal{~}

\definecolor{orange}{rgb}{0.7,0.3,0}
\definecolor{blue}{rgb}{.2,.6,.75}
\definecolor{green}{rgb}{.4,.7,.4}
\definecolor{purple}{RGB}{127,0,255}

\newtheorem{theorem}{Theorem}[section]
\newtheorem{corollary}[theorem]{Corollary}
\newtheorem{definition}{Definition}
\newtheorem{lemma}[theorem]{Lemma}

\newcommand{\abs}[1]{\lvert #1 \rvert}

\begin{document}

\begin{frontmatter}

\title{Zeros of derivatives of $L$-functions in the Selberg class on $\Re(s)<1/2$}

\author[A]{Sneha Chaubey}
\ead{sneha@iiitd.ac.in}
\address[A]{Department of Mathematics,\\ Indraprastha Institute of Information Technology,\\ New Delhi - 110020, India \vspace{0.5 cm}}
\author[B]{Suraj Singh Khurana}
\ead{suraj.singh.khurana@gmail.com}
\author[C]{Ade Irma Suriajaya}
\address[B]{Department of Mathematics and Statistics,\\ Indian Institute of Technology Kanpur,\\ Uttar Pradesh – 208016, India \vspace{0.5 cm}}
\address[C]{Faculty of Mathematics, \\ Kyushu University, 744 Motooka, Nishi-ku,\\ Fukuoka 819-0395, Japan}
\ead{adeirmasuriajaya@math.kyushu-u.ac.jp}

\begin{abstract}
In this article, we show that the Riemann hypothesis for an $L$-function $F$ belonging to the Selberg class implies that all the derivatives of $F$ can have at most finitely many zeros on the left of the critical line with imaginary part greater than a certain constant. This was shown for the Riemann zeta function by Levinson and Montgomery in 1974.
\end{abstract}


\end{frontmatter}


\section{Introduction and main results}
It is well-known that the Riemann zeta function $\zeta(s)$ is one of the most famous and mysterious functions. It has its origin from the well-known memoir of Riemann in the 19\textsuperscript{th} century.
Zeros of this function capture the distribution of prime numbers and thus its distribution is of great interest in number theory.
Zeros of $\zeta(s)$ are classified into two groups: trivial and non-trivial zeros. As their names suggest, trivial zeros of $\zeta(s)$ are precisely known whereas the exact location of non-trivial zeros remains one of the most important math problems.
It is conjectured by Riemann in the same paper \cite{rie} that all non-trivial zeros of $\zeta(s)$ lie on the line $\Re(s)=1/2$, known as the Riemann hypothesis.
Using basic properties of $\zeta(s)$, it can be shown that non-trivial zeros lie in the critical strip $0\le\Re(s)\le1$, and together with the well-known prime number theorem, one can remove the equalities.
These  properties of $\zeta(s)$ are expected to be the essence of a meromorphic function satisfying the Riemann hypothesis.
To investigate this in a general setup, 
Selberg \cite{MR1220477} defined a class of $L$-functions with properties similar to $\zeta(s)$, and which therefore is known as the Selberg class.
\begin{definition}
The Selberg class $\mathcal{S}$ is defined as the set of all Dirichlet series
\begin{equation} \label{DSS}
F(s)=\sum_{n=1}^\infty \frac{a_n}{n^s}
\end{equation}
that satisfies the following five axioms:
\begin{enumerate}
    \item The Dirichlet series converges absolutely for $\Re(s)>1$.
    \item There exists an integer $m_f \geq 0$ such that the function $(s-1)^m F(s)$ is entire of finite order.
    \item $F$ satisfies the following functional equation
    \begin{equation}
        \Phi(s)=\omega_{F}\overline{\Phi}(1-s),
    \end{equation}
    where
    \begin{equation} \label{FEQ}
    \Phi(s)=Q^s\prod\limits_{j=1}^{r}\Gamma(\lambda_js+\mu_j)F(s)=\gamma(s)F(s),
    \end{equation}
    and $r\geq 0$, $Q>0$, $\lambda_j >0$, $\Re(\mu_j)\geq 0$, $\abs{\omega_F}=1$ are parameters depending on $F$.
    \item The coefficients $a(n)\ll n^\epsilon $ for every $\epsilon \geq 0$.
    \item For $\Re(s)>1$, we have
    \begin{equation} \label{EPF}
        \log F(s)= \sum\limits_{n=1}^{\infty}\dfrac{b(n)}{n^s},
    \end{equation}
    where $b(n)=0$ unless $n=p^m$ with $m \geq 1$, and $b(n) \ll n^\theta$ for some $\theta < 1/2$.
\end{enumerate}
\end{definition}
Some well-known examples of $L$-functions in the Selberg class are Dirichlet $L$-functions associated to primitive characters, Dedekind zeta functions of number fields $K/\mathbb{Q}$, Hecke $L$-functions with primitive characters, Artin $L$-functions associated to irreducible Galois representations of number fields $K/\mathbb{Q}$, and $L$-functions associated to holomorpic newforms.
For a survey of results and open problems in the theory of Selberg class one may refer to the well-known articles of Perelli and Kaczorowski \cite{MR2175035,MR2128842,MR2277660}. Naturally, one can ask if the properties satisfied by the Riemann zeta function are also true for $L$-functions belonging to the Selberg class $\mathcal{S}$. Most often, such results are difficult to establish, for example, it was only recently proven that the prime number theorem for $L$-functions in $\mathcal{S}$ is equivalent to the non-vanishing of the $L$-function on the line $\Re{s}=1$ \cite{MR1981179}. Although this non-vanishing criterion for the Riemann zeta function has been independently proven by Hadamard and de la Vall\'ee Poussin in 1896, the general case for $L$-functions in $\mathcal{S}$ is still open. Another famous result, known as Speiser's criterion, relates non-trivial zeros of the Riemann zeta function to those of its first derivative $\zeta'(s)$. More precisely, it states that the Riemann hypothesis is equivalent to the absence of non-real (we refer to them as non-trivial) zeros for $\zeta'(s)$ for $\Re(s)<1/2$. Consequently, Speiser's criterion \cite{MR1512953} for the Riemann hypothesis has led to a growing interest in the study of zeros of derivatives of the Riemann zeta function, and in particular, in establishing analogous criteria for other $L$-functions in general \cite{MR2026098,MR3386244,MR3724167,MR4038389}.

In this note, we establish the following results for $L$-functions in $\mathcal{S}$. 
\begin{theorem} \label{SPR1}
For a given $F\in\mathcal{S}$, let $T_F$ be a real number given by
\begin{equation} \label{T_F}
T_F := \max\limits_{j} \left\vert \dfrac{\Im{\mu_j}}{\lambda_j}\right\vert,
\end{equation}
where $\mu_j$ and $\lambda_j$ are as in \eqref{FEQ}.
There exists a constant $\alpha_{F,k}>0$ such that $F^{(k)}(s)$ has no zeros in the region $\Re(s)<-\alpha_{F,k}$ and $|\Im(s)|>T_F$. Further, in the region $|\Im(s)|\leq T_F$, $F^{(k)}(s)$ has a zero $z^{(k)}_{n,j}$ in each rectangle $\mathcal{R}_{n,j}$ with vertices $\dfrac{\pm c_k-2\mu_j-2n}{2\lambda_j} \pm it_{F,k}$, where $1 \leq j \leq r$, $n\ge N_{F,k}$ runs over all positive integers such that all these rectangles lie in the region $\Re(s)<-\alpha_{F,k}$, 
that is, for each $k$, there is a positive integer $N_{F,k}$ such that
$$
\max_{1 \leq j \leq r} \dfrac{\pm c_k-2\mu_j-2N_{F,k}}{2\lambda_j} < -\alpha_{F,k},
$$
and $c_k,t_{F,k}>0$ are small positive constants depending on $F$ and $k$. In particular, for each $j$ we have
\begin{equation*}
    z^{(k)}_{n,j} = -\dfrac{n}{\lambda_j}+O_k(1).
\end{equation*}
\end{theorem}
The above result gives a zero-free region for $F^{(k)}(s)$ on the left half plane and shows the approximate location of zeros close to the real line when $\Re(s)<-\alpha_{F,k}$ for some constant $\alpha_{F,k}$. This also shows that those zeros, denoted by $z^{(k)}_{n,j}$ in the theorem, can be regarded as trivial zeros of $F^{(k)}(s)$. Analogous results for derivatives of the Riemann zeta function $\zeta(s)$ have been proven by Spira \cite{MR181621,MR0263754}, and for derivatives of Dirichlet $L$-functions $L(s,\chi)$ by Y\i ld\i r\i m \cite{MR1432879}, with improvements by Akatsuka and the third author \cite{MR3724167} for the first derivative.
Our second result is as follows.

\begin{theorem} \label{Main}
Let $F\in\mathcal{S}$.
For any non-negative integer $k$, if there exists a constant $T_{F,k}>0$, depending on $F$ and $k$, such that $F^{(k)}(s)$ has only a finite number of zeros with $\Re(s) < 1/2$ and $|\Im(s)|>T_{F,k}$, then the same is true for $F^{(k+l)}(s)$ for all positive integers $l$. 
\end{theorem}
\noindent
This leads to the following result.
\begin{corollary}
The Riemann hypothesis for an $L$-function $F$ in the Selberg class, implies that there are only finitely many zeros to the left of the critical line and away from the real line of any derivatives $F^{(k)}$ of $F$.
\end{corollary}
\noindent
Analogous results have been proven by Levinson and Montgomery in 1974 \cite[Theorem 7]{MR417074} for $\zeta(s)$, and by Y\i ld\i r\i m in 1996 \cite{MR1432879} for $L(s,\chi)$. 

\medskip

Our aim is to investigate the number of zeros and distribution of real parts of zeros of derivatives of $L$-functions $F(s)$ in $\mathcal{S}$, analogous to Berndt's work in \cite{MR0266874}, and Levinson and Montgomery's in \cite[Theorem 10]{MR417074}.
Conditional results under the Riemann hypothesis include Akatsuka's work \cite{MR2944752} on $\zeta'(s)$, and the third author's work \cite{MR3402773} on $\zeta^{(k)}(s)$ for all positive integers $k$, with improvements by Ge \cite{MR3658174} and his work with the third author \cite{MR4260166}. Analogous results for $L^{(k)}(s,\chi)$ were proven by Y\i ld\i r\i m \cite{MR1432879} and the third author \cite{MR3682476} with improvements for $k=1$ by Akatsuka and the third author \cite{MR3724167}, and by Ge \cite{MR3984262}. These are to be done in our subsequent paper.

\section{Notation}

\begin{enumerate}[label=(\roman*)]
\item Throughout this paper, we use the variable $s=\sigma+it$ to denote a complex number with real part $\sigma$ and imaginary part $t$. The real part of any given complex value $z$ is denoted by $\Re{(z)}$ or $\Re{z}$ and the imaginary part by $\Im{(z)}$ or $\Im{z}$. 
\item $F$ is always used to denote an element of the Selberg class $\mathcal{S}$.
\item For any function $f$, the function $\overline{f}$ is defined by $\overline{f}(s)=\overline{f(\overline{s})}$, where $\overline{z}$ denotes the complex conjugate of $z$.
\item For an entire function $f(z)=\sum_{n=0}^\infty a_n z^n$ of finite order (of growth) \cite[p. 138]{MR1976398}, we denote its order by $\rho(f)$ which can be computed as
\begin{equation} \label{DGE}
\rho(f)=\limsup_{n\to\infty}\frac{n\ln n}{-\ln|a_n|},
\end{equation}
\cite[p. 74]{MR0447572}.
\item The multinomial coefficients are given by
$$ \binom{n}{k_1,k_2,\cdots,k_l}=\dfrac{n!}{k_1!k_2!\cdots k_l!}. $$
\end{enumerate}

\section{Preliminaries}

In this section, we will review several important results which will enable us to derive our main theorem. Some of these results can be obtained from earlier works, while some require proofs.
\begin{lemma} \label{SS03}
Let $f$ be an entire function of order $\leq \rho$. If $z_1,z_2,\cdots$ denote the zeros of $f$, with $z_k \neq 0$, then for all $s>\rho$, we have
\begin{equation*}
    \sum\limits_{k=1}^{\infty}\dfrac{1}{|z_k|^{s}}<\infty.
\end{equation*}
\end{lemma}
\begin{proof}
See \cite[Chapter 5, Theorem 2.1]{MR1976398}.
\end{proof}
The next important result provides a uniform estimate for the size of $L$-function in the Selberg class.
\begin{lemma} \label{JSL}
Let $F \in \mathcal{S}$. For $t \geq 1$, uniformly in $\sigma$,
\begin{equation}
    \Delta_{F}(\sigma+it)=(\lambda Q^2t^{d_F})^{1/2-\sigma-it}\exp\left(itd_F+\dfrac{i\pi(\mu-d_F)}{4}\right)\left(\omega_F+O\left(\dfrac{1}{t}\right)\right),
\end{equation}
where
\begin{equation*}
    \Delta_{F}(s)=\omega_FQ^{1-2s}\prod\limits_{j=1}^{r}\dfrac{\Gamma(\lambda_{j}(1-s)+\overline{\mu_j})}{\Gamma(\lambda_js+\mu_j)},
\end{equation*}
$d_F=2\sum\limits_{j=1}^{r}\lambda_j$, $\mu=2\sum\limits_{j=1}^{r}(1-2\mu_j)$ and $\lambda=\prod\limits_{j=1}^{r}\lambda_{j}^{2\lambda_j}$.
\end{lemma}
\begin{proof}
See \cite[Lemma 6.7]{MR2330696}.
\end{proof}

Next we prove five key lemmas. Firstly, we provide an estimate on the size of the entire function $(s-1)^mF(s)$.
\begin{lemma} \label{ELF}
Let $F \in \mathcal{S}$. Then we have
\begin{equation*}
(s-1)^mF(s) \ll \begin{cases}
                    1 & \text{ if}~  -1 \leq \Re{s} \leq 2, |\Im{s}| \leq 1,  \\
                    |s|^m & \text{ if}~~~~ \Re{s}\geq 2, \\
                    e^{C|s|\log|s|} & \text{ if}~ -1 \leq \Re{s} \leq 2, |\Im{s}| \geq 1; ~\text{or }~ \Re{s} \leq -1,
                 \end{cases}
\end{equation*}
for some $C>0$ depending on $F$.
\begin{proof}
Let $0<\alpha\le1$ be any fixed real number. It follows from the uniform convergence of the Dirichlet series expression \eqref{DSS} for $F(s)$, that $F(s) \ll_{\alpha}1$ whenever $\Re(s) \geq 1+ \alpha$. Therefore we have $(s-1)^mF(s) \ll |s|^m$ on the half plane $\Re{s \geq 2}$.
In the case when $s$ is lying in the bounded region $-\alpha \leq \Re{s} \leq 1+\alpha, |\Im{s}| \leq 1$ , the entire function $(s-1)^mF(s)$ is bounded. Therefore we have $(s-1)^mF(s) \ll_\alpha 1 $. For the case when $s$ is lying in the region $-\alpha \leq \Re{s} \leq 1+\alpha, |\Im{s}| \geq 1$, with the help of Lemma \ref{JSL} we get $F(s) \ll e^{(|s|+1)d_F\log |s|}$. This gives us the desired estimate for $(s-1)^mF(s)$ in the region $-1 \leq \Re{s} \leq 2, |\Im{s}| \geq 1$. In the remaining case when $\Re{s} \leq -1$, observe that whenever $\Re{s}\leq -\alpha$, we have $|\lambda_{j}(1-s)+\overline{\mu_j}|>0$ for $1\leq j \leq r$. Therefore we can use Stirling estimates for $\Gamma(\lambda_{j}(1-s)+\overline{\mu_j})$ and the functional equation to estimate
\begin{equation*}
    F(s)=w_FQ^{1-2s}\prod\limits_{j=1}^{r}\dfrac{\Gamma(\lambda_{j}(1-s)+\overline{\mu_j})}{\Gamma(\lambda_js+\mu_j)}\overline{F(1-\overline{s})}
\end{equation*}
and get $F(s)\ll e^{C|s|\log|s|}$ for some positive $C$ depending on $F$. This gives us the desired estimate for the function $(s-1)^mF(s)$.
\end{proof}
\end{lemma}
In the next two sections, we complete the proofs of our main results.

\section{Zeros of $F^{(k)}(s)$}

In this section, we prove Theorem \ref{SPR1} and establish relevant information about the zeros of the derivatives of $F$ on the left half plane. Additionally, we also use these results to obtain a zero-free region on the right half plane and deduce that $F$ and all its derivatives are of order $1$.
\begin{proof}[Proof of Theorem \ref{SPR1}]
With the help of functional equation for $F(s)$, and reflection formula for $\Gamma(s)$, we write
\begin{align*}
    \overline{F}(1-s)&=\dfrac{\gamma(s)F(s)}{w_F\overline{\gamma}(1-s)} \\
    &=\dfrac{Q^{2s-1}\prod\limits_{j=1}^{r}\Gamma(\lambda_js+\mu_j)\prod\limits_{j=1}^{r}\Gamma(1-(\lambda_j(1-s)+\overline{\mu_j}))}{w_F\prod\limits_{j=1}^{r}\Gamma(\lambda_j(1-s)+\mu_j)\prod\limits_{j=1}^{r}\Gamma(1-(\lambda_j(1-s)+\overline{\mu_j}))}F(s) \\
    &=: F_0(s)F_1(s)F_2(s),
\end{align*}
where
\begin{equation*}
F_0(s)= F(s)\dfrac{Q^{2s-1}}{w_F\pi^{r}},
\end{equation*}
\begin{equation*}
F_1(s)=\prod\limits_{j=1}^{r}\Gamma(\lambda_js+\mu_j)\Gamma(1-(\lambda_j(1-s)+\overline{\mu_j})),
\end{equation*}
and
\begin{equation*}
F_2(s)=\prod\limits_{j=1}^{r}\sin(\pi(\lambda_j(1-s)+\overline{\mu_j})).
\end{equation*}
Differentiating $k$ times we obtain
\begin{align*}
    \overline{F}^{(k)}(1-s) := \dfrac{d^k}{ds^k}\overline{F}(1-s) &= \sum\limits_{\substack{k_0+k_1+k_2=k}}\binom{n}{k_0,k_1,k_2}F_0^{(k_0)}(s)F_1^{(k_1)}(s)F_2^{(k_2)}(s) \\
    &=F_1^{(k)}(s)F_0(s)F_2(s)+\sum\limits_{\substack{k_0+k_1+k_2=k \\ k_1 \leq k-1}}\binom{n}{k_0,k_1,k_2}F_0^{(k_0)}(s)F_1^{(k_1)}(s)F_2^{(k_2)}(s)
\end{align*}
By the triangle inequality we have
\begin{equation*}
    \left|\overline{F}^{(k)}(1-s)\right| \geq \left|F_1^{(k)}(s)F_0(s)F_2(s)\right| - \left|\sum\limits_{\substack{k_0+k_1+k_2=k \\ k_1 \leq k-1}}\binom{n}{k_0,k_1,k_2}F_0^{(k_0)}(s)F_1^{(k_1)}(s)F_2^{(k_2)}(s)\right|.
\end{equation*}
We now investigate when the right hand side of the above expression is positive. We first evaluate $F_1^{(k)}(s)$. Using the work of Spira \cite[p. 679]{MR181621}, it is not difficult to show
\begin{align*}
    F_1^{(k)}(s)&=\prod\limits_{j=1}^{r}\Gamma(\lambda_js+\mu_j)\Gamma(1-(\lambda_j(1-s)+\overline{\mu_j})) \\
    & \times \sum\limits_{n_1+\cdots+n_{2r}=k}\binom{k}{n_1,\cdots,n_{2r}}\lambda_{1}^{n_1+n_{r+1}}\cdots\lambda_{r}^{n_r+n_{2r}}\\
    &\times\left(\log^{n_1}(\lambda_1s+\mu_1)+\sum\limits_{n=0}^{n_1-1}E_{nn_1}(\lambda_1s+\mu_1)\log^{n}(\lambda_1s+\mu_1)\right) \\
    &\times \cdots \\
    &\times \Biggl(\log^{n_2r}(1-(\lambda_{r}(1-s)+\overline{\mu_r})\\
    &~~~ +\sum\limits_{n=0}^{n_{2r}-1}E_{nn_{2r}}(1-(\lambda_{r}(1-s)+\overline{\mu_r})\log^{n}(1-(\lambda_{r}(1-s)+\overline{\mu_r})\Biggr)
\end{align*}
where $E_{nn_i}$ are polynomials satisfying $E_{nn_{i}}=O(|s|^{-1})$. The major contribution in the growth of $F_1^{(k)}(s)F_0(s)F_2(s)$ comes from $G(s)\tau_k(\log{s})^k$, where
\begin{equation*}
G(s)=F_0(s)F_2(s)\prod\limits_{j=1}^{r}\Gamma(\lambda_js+\mu_j)\Gamma(1-(\lambda_j(1-s)+\overline{\mu_j}))
\end{equation*}
and 
\begin{equation*}
    \tau_k=\sum\limits_{n_1+\cdots+n_{2r}=k}\binom{k}{n_1,\cdots,n_{2r}}\lambda_{1}^{n_1+n_{r+1}}\cdots\lambda_{r}^{n_r+n_{2r}}.
\end{equation*}
Now taking out the factor $G(s)\tau_k(\log{s})^{k-1}$ from $F_1^{(k)}(s)$, we obtain the desired inequality
\begin{equation*}
    \left\vert \log s + \dfrac{1}{G(s)\tau_k\log^{k-1} s}\left(\sum\limits_{n=0}^{k-1}H_n(s)\log^n s\right) \right\vert > \left\vert\sum\limits_{\substack{k_0+k_1+k_2=k \\ k_1 \leq k-1}}\binom{n}{k_0,k_1,k_2}\dfrac{F_0^{(k_0)}(s)F_1^{(k_1)}(s)F_2^{(k_2)}(s)}{G(s)\tau_k\log^{k-1}}\right\vert,
\end{equation*}
where $H_n(s)$ are functions satisfying $H_n(s)=O(|s|^{-1})$. The above inequality is satisfied provided
\begin{equation} \label{I1}
    |\log s|>\left\vert \dfrac{1}{G(s)\tau_k\log^{k-1} s}\left(\sum\limits_{n=0}^{k-1}H_n(s)\log^n s\right) \right\vert + \left\vert\sum\limits_{\substack{k_0+k_1+k_2=k \\ k_1 \leq k-1}}\binom{n}{k_0,k_1,k_2}\dfrac{F_0^{(k_0)}(s)F_1^{(k_1)}(s)F_2^{(k_2)}(s)}{G(s)\tau_k\log^{k-1}}\right\vert.
\end{equation}
Now observe that for $|t|>T_F$, both the first and the second sums on the right hand side above are bounded if $\sigma$ is sufficiently large. Thus the inequality above is satisfied whenever $\sigma > \alpha_1+1$ for some large $\alpha_1>0$, which clearly depends on $F$ and $k$. This gives a zero-free region for $\overline{F}^{(k)}(1-s)$, which implies that $F^{(k)}(s) \neq 0$ for $\sigma < - \alpha_1$ and $|t|>T_F$.

Next we investigate zeros in $|t|\leq T_F$.
We first note that $G(s)\tau_k(\log{s})^k$ only have zeros coming from the sine factors $F_2(s)$.
Further, zeros of $F_2(s)$ each lies in the rectangle $R_{n,j}$ with vertices $r_1,r_2,r_3,r_4$ given by
$$r_1=1+\frac{2\overline{\mu_j}+2n - c_k}{2\lambda_j}+t_{F,k}i,$$
$$r_2=1+\frac{2\overline{\mu_j}+2n - c_k}{2\lambda_j}-t_{F,k}i,$$
$$r_3=1+\frac{2\overline{\mu_j}+2n + c_k}{2\lambda_j}-t_{F,k}i,$$
$$r_4=1+\frac{2\overline{\mu_j}+2n + c_k}{2\lambda_j}+t_{F,k}i,$$
for $1\leq j \leq r$ and every positive integer $n$
where $t_{F,k}>0$ and $c_k>0$ are chosen small enough so that no rectangles intersect each other. This is always possible unless
\begin{equation} \label{CM}
\frac{\overline{\mu_j}}{\lambda_j}=\frac{\overline{\mu_l}}{\lambda_l}
\end{equation}
for some $j\neq l$, in which case, we count the zeros with multiplicity. Observe also that on the boundaries of all the rectangles $R_{n,j}$, the inequality \eqref{I1} holds provided $\sigma$ is large enough, say $\sigma > \alpha_2+1$, where $\alpha_2$ also depends on $F$ and $k$. Now by Rouche's theorem and inequality \eqref{I1}, we see that the number of zeros of $\overline{F}^{(k)}(1-s)$ is equal to that of $G(s)\tau_k(\log{s})^k$, provided $\sigma > \alpha_2+1$. Moreover these zeros of $\overline{F}^{(k)}(1-s)$ are located in the rectangles $R_{n,j}$.
Finally we set $\alpha_{F,k} := \max\{\alpha_1,\alpha_2\}$ and take the smallest positive integer $N_{F,k}$ satisfying
$$
\min_{1\leq j \leq r} \left(1+\frac{2\overline{\mu_j}+2N_{F,k}-c_k}{2\lambda_j} \right) > \alpha_{F,k},
$$
which both depend on $F$ and $k$.
Therefore in the region $|t|\le T_F$, we have for each $1\le j\le r$ and $n\ge N_{F,k}$, zeros $z_{n,j}$ of $F^{(k)}(s)$ satisfy
\begin{equation*}
     z_{n,j} = -\dfrac{n}{\lambda_j}+O(1).
\end{equation*}
\end{proof}

The next lemma provides a zero-free region on the right half plane for the derivatives of $F$.
\begin{lemma} \label{ZFR}
There exists a constant $\mathcal{A}_{F,k}>0$ such that $F^{(k)}(s)$ has no zeros in the region $|\sigma|>\mathcal{A}_{F,k}$ and $|t|>T_F$, where $T_F$ is as determined in \eqref{T_F}.
\end{lemma}
\begin{proof}
This follows from \cite[Proposition 1]{MR4230681} and Theorem \ref{SPR1}.
\end{proof}

The following lemma establishes the order of the entire function $(s-1)^mF(s)$ which is crucial in our discussion. The proof is motivated by the observations reported in \cite{NTSLI}.
\begin{lemma}\label{WDO}
The entire function $(s-1)^mF(s)$ is of order 1.
\end{lemma}
\begin{proof}
We first show that the  order of $(s-1)^mF(s)$ is greater than or equal to 1. For this, we recall that for each $j$, $F(s)$ has equally spaced trivial zeros at $s=-(\mu_j+l)/\lambda_j$ where $l \in \mathbb{Z}_{\geq 0}$, which are nothing but the poles of the gamma factors appearing in \eqref{FEQ}. Now, if the order of $(s-1)^mF(s)$ is strictly less than 1, then by Theorem \ref{SS03}, we would have
\begin{equation*}
    \sum\limits_{\substack{s \neq 0 \\ F(s)=0}}\dfrac{1}{|s|^{1-\epsilon}} < \infty
\end{equation*}
which is a contradiction.
Therefore, it suffices to show that $(s-1)^mF(s)$ is of order $\leq 1$. This follows from Lemma \ref{ELF}.
\end{proof}
As a consequence of the above lemma, we can compute the order of the $k$-th derivative as stated below. 
\begin{lemma} \label{DFO}
For any non-negative integer $k$, $(s-1)^{k+m}F^{(k)}(s)$ is an entire function of order 1.
\end{lemma}
\begin{proof}
The case $k=0$ follows from Lemma \ref{WDO}. For $k \geq 1$, consider the Laurent series expansion of $F(s)$ at $s=1$ written as
\begin{equation*}
    F(s)=\sum\limits_{n=-m}^{\infty}a_n(s-1)^n.
\end{equation*}
Differentiating $k$ times we get
\begin{equation*}
    F^{(k)}(s)=\sum\limits_{n=-m}^{\infty}n(n-1)\cdots (n-k+1) a_n(s-1)^{n-k}.
\end{equation*}
This gives
\begin{align*}
    (s-1)^{k+m}F^{(k)}(s)&=\sum\limits_{n=-m}^{\infty}n(n-1)\cdots (n-k+1) a_n(s-1)^{n+m}\\
    &=\sum\limits_{N=0}^{\infty}(N-m)(N-m-1)\cdots(N-m-k+1)a_{N-m}(s-1)^N.
\end{align*}
Now the order of entire function $(s-1)^{k+m}F^{(k)}(s)$ is equal to
\begin{align*}
&\limsup_{N\to\infty}\frac{N\log N}{-\log|(N-m)(N-m-1)\cdots(N-m-k+1)a_{N-m}|} \\
=&\limsup_{N\to\infty}\dfrac{N\log N}{-\log|a_{N-m}|-\log|N^{k}(1-m/N)(1-(m+1)/N)\cdots(1-(m+k-1)/N|} \\
=&\limsup_{N\to\infty}\dfrac{N\log N}{-\log|a_{N-m}|\left(1+\log|\dfrac{N^{k}(1-m/N)(1-(m+1)/N)\cdots(1-(m+k-1)/N}{a_{N-m}}|\right)} \\
=&\limsup_{N\to\infty}\dfrac{N\log N}{-\log|a_{N-m}|\left(1+\log|\dfrac{N^{k}(1-m/N)(1-(m+1)/N)\cdots(1-(m+k-1)/N}{a_{N-m}}|\right)} \\
\leq & \limsup_{N\to\infty}\dfrac{N\log N}{-\log|a_{N-m}|} \leq 1,
\end{align*}
where the last inequality follows from Lemma \ref{WDO}. Now from Theorem \ref{SPR1}, we know that 
\begin{equation*}
    \sum\limits_{\substack{z \\ F^{(k)}(z)=0}}\dfrac{1}{|z|^{1-\epsilon}} 
\end{equation*}
diverges for any $\epsilon > 0$. Therefore, by Theorem \ref{SS03} we conclude that the order of $(s-1)^{k+m}F^{(k)}(s)$ is equal to 1.
\end{proof}

\section{Proof of Theorem \ref{Main}}

By Lemma \ref{DFO}, we know that $(s-1)^{k+1}F^{(k)}(s)$ is of finite order. Let $A=\{A_\ell\}_{\ell\geq 1}$ and $m_0$ denote the set of zeros, and the order of zero at $s=0$ of the function $(s-1)^{k+1}F^{(k)}(s)$, respectively. Using Hadamard factorization theorem, we write
\begin{equation*}
    (s-1)^{k+m}F^{(k)}(s)=s^{m_0}e^{g(s)}\prod\limits_{\ell=1}^{\infty}E_{1}\left(\dfrac{s}{A_\ell}\right),
\end{equation*}
where $g$ is a polynomial of ${\rm deg}(g)\leq 1$, and $E_{1}(s)=(1-s)e^s$. Taking logarithmic derivative, we obtain 
\begin{equation} \label{logderi1}
    \dfrac{F^{(k+1)}}{F^{(k)}}(s)=\dfrac{m_0}{s}-\dfrac{k+m}{s-1}+\sum\limits_{\ell=1}^{\infty}\left(\dfrac{1}{s-A_\ell}+\dfrac{1}{A_\ell}\right)+O(1).
\end{equation}
Utilizing Theorem \ref{SPR1}, we write the set $\{A_\ell\}_\ell$ as a disjoint union of $B_1(:=B_{1,0}\cup B_{1,1}),B_2$ and $B_3$, where
\begin{equation*}
    B_{1,0}=\{ z_{n,j} \mid n \geq N_{F,k}, 1 \leq j \leq r\},
\end{equation*}
\begin{equation*}
    B_{1,1}=  \{A_\ell \mid \Re(A_\ell)>-\alpha_{F,k}, |\Im(A_\ell)|\leq T_F \},
\end{equation*}
\begin{equation*}
    B_2=\{A_\ell|A_\ell \notin B_1, \Re{(A_\ell)} \geq 1/2\}
\end{equation*}
and
\begin{equation*}
    B_3=\{A_\ell|A_\ell \notin B_1, \Re{(A_\ell)} < 1/2\}.
\end{equation*}
Here $T_F$ is as defined in \eqref{T_F}.

Using this partition of the set $\{A_\ell\}_\ell$, we write the sum on the right hand side of \eqref{logderi1} as
\begin{equation}
    \sum\limits_{\ell=1}^{\infty}\left(\dfrac{1}{s-A_\ell}+\dfrac{1}{A_\ell}\right)=\sum\limits_{i=1}^{3}\sum\limits_{a \in B_i}\left(\dfrac{1}{s-a}+\dfrac{1}{a}\right).
\end{equation}
From \cite[Proposition 1]{MR4230681}, it is easy to see that the set $B_{1,1}$ is finite.
Our assumption implies that the sum over $B_3$ is finite. Moreover, applying Lemma \ref{ZFR}, we find that
\begin{equation*}
    \max\limits_{a\in B_2 \cup B_3} |\Re(a)|<\mathcal{A}_{F,k}.
\end{equation*}
Hence,
\begin{align*}
    \sum\limits_{\ell=1}^{\infty}\Re{\left(\dfrac{1}{s-A_\ell}+\dfrac{1}{A_\ell}\right)}&=\sum\limits_{a \in B_{1,0}}\Re{\left(\dfrac{1}{s-a}+\dfrac{1}{a}\right)}+ \sum\limits_{a \in B_2}\Re{\left(\dfrac{1}{s-a}+\dfrac{1}{a}\right)}+ O(1)\\
    &=\sum\limits_{a \in B_{1,0}}\left(\dfrac{\sigma-\Re{a}}{|s-a|^2}+\dfrac{\Re{a}}{|a|^2}\right)+\sum\limits_{a \in B_2}\left(\dfrac{\sigma-\Re{a}}{|s-a|^2}+ \dfrac{\Re{a}}{|a|^2}\right) + O(1),
\end{align*}
where the implied constants above depend on $F$ and $k$ since the number of zeros in $B_{1,1}$ depends on $\alpha_{F,k}$ and $\mathcal{A}_{F,k}$, and the assumption is for any fixed $F$ and $k$ which results in finiteness of zeros in $B_3$.

We estimate the sum over zeros in $B_2$ as  
\begin{equation*}
    \sum\limits_{a \in B_2}\left(\dfrac{\sigma-\Re{a}}{|s-a|^2}+ \dfrac{\Re{a}}{|a|^2}\right) < \sum\limits_{a \in B_2}\dfrac{\frac{1}{2}-\Re{a}}{|s-a|^2}+\mathcal{A}_{F,k}\sum\limits_{a \in B_2}\dfrac{1}{|a|^2}
    \le \mathcal{A}_{F,k}\sum\limits_{a \in B_2}\dfrac{1}{|a|^2} = O(1).
\end{equation*}
The sum over $B_{1,0}$ can be expanded as follows
\begin{align*}
  \sum\limits_{a \in B_{1,0}}\left(\dfrac{\sigma-\Re{a}}{|s-a|^2}+ \dfrac{\Re{a}}{|a|^2}\right)&= \sum\limits_{a \in B_{1,0}} \left(\dfrac{\sigma}{|s-a|^2}+\dfrac{|s|^2\Re{a}}{|a|^2|s-a|^2} -\dfrac{2\sigma (\Re{a})^2}{|a|^2|s-a|^2} -\dfrac{2t\Re{a}\Im{a}}{|a|^2|s-a|^2} \right) \\
  &=\sum\limits_{a \in B_{1,0}}\left(\dfrac{\sigma}{|s-a|^2}\left[1-\dfrac{2(\Re{a})^2}{|a|^2}\right] + \dfrac{|s|^2\Re{a}}{|a|^2|s-a|^2}-\dfrac{2t\Re{a}\Im{a}}{|a|^2|s-a|^2}\right).
\end{align*}
When $-\mathcal{A}_{F,k}<\sigma<1/2$ and $t$ is sufficiently large, both the sums
\begin{equation*}
    \sum\limits_{a \in B_{1,0}}\dfrac{\sigma}{|s-a|^2}\left[1-\dfrac{2(\Re{a})^2}{|a|^2}\right]
\end{equation*}
and 
\begin{equation*}
    \sum\limits_{a \in B_{1,0}} \dfrac{2t\Re{a}\Im{a}}{|a|^2|s-a|^2}
\end{equation*}
are bounded. For the remaining sum, since $\Re{a}<0$ for all $a \in B_{1,0}$, it suffices to consider the region $|a|<|s|/2$. The sum can then be estimated as
\begin{align*}
    \sum\limits_{a \in B_{1,0}}\dfrac{|s|^2\Re{a}}{|a|^2|s-a|^2} &\leq \sum\limits_{j=1}^{r}\sum\limits_{|a_{n,j}|<|s|/2}\dfrac{|s|^2\Re{a_{n,j}}}{|a_{n,j}|^2|s-a_{n,j}|^2} \\
    &\leq \dfrac{4}{9}\sum\limits_{j=1}^{r}\sum\limits_{n}\dfrac{1}{\Re{a_{n,j}}}+O(1)\\
    &\leq -K\log|s|+O(1)
\end{align*}
for some $K>0$. Thus $F^{(k+1)}(s)\neq 0$ for sufficiently large $t$, which was to be shown. The result now follows using induction on $k$.
\section*{Acknowledgements}
Research of the first author was supported by the Science and Engineering Research Board, Department of Science and Technology, Government of India under grant SB/S2/RJN-053/2018. The second author acknowledges the support of National Board for Higher Mathematics (NBHM), Department of Atomic Energy (DAE), Government of India (DAE Ref no: 0204/37/2021/R\&D-II/15564). Research of the third author was supported by JSPS KAKENHI Grant Number 18K13400 and MEXT Initiative for Realizing Diversity in the Research Environment.

\bibliographystyle{alpha}
\bibliography{mybibfile}
\end{document}